\documentclass[a4, 11pt]{amsart}
\usepackage{amssymb,latexsym}
\usepackage{color}
\usepackage[left=35mm, right=35mm, top=30mm, bottom=30mm, includefoot, includehead]{geometry}
\theoremstyle{plain}
\newtheorem{theorem}{Theorem}[section]
\newtheorem{corollary}[theorem]{Corollary}
\newtheorem{proposition}[theorem]{Proposition}
\newtheorem{lemma}[theorem]{Lemma}
\theoremstyle{definition}
\newtheorem{definition}[theorem]{Definition}

\newtheorem{remark}[theorem]{Remark}
\newtheorem{question}[theorem]{Question}

\newcommand{\enm}[1]{\ensuremath{#1}}          %

\newcommand{\cal}[1]{\mathcal{#1}}

\newcommand{\NN}{\enm{\mathbb{N}}}

\newcommand{\ZZ}{\enm{\mathbb{Z}}}

\newcommand{\PP}{\enm{\mathbb{P}}}

\newcommand{\Aa}{\enm{\cal{A}}}
\newcommand{\Bb}{\enm{\cal{B}}}

\newcommand{\Ee}{\enm{\cal{E}}}

\newcommand{\Ii}{\enm{\cal{I}}}

\newcommand{\Ll}{\enm{\cal{L}}}

\newcommand{\Oo}{\enm{\cal{O}}}

\renewcommand{\phi}{\varphi}
\renewcommand{\theta}{\vartheta}
\renewcommand{\epsilon}{\varepsilon}

\renewcommand{\to}[1][]{\xrightarrow{\ #1\ }}

\newcommand{\old}[1]{}

\date{}
\title[Minimal degree equations for curves and surfaces]{Minimal degree equations for curves and surfaces \\ (variations on a theme of Halphen)}

\author{Edoardo Ballico and Emanuele Ventura}

\address{Universit\`a di Trento, 38123 Povo (TN), Italy}
\email{edoardo.ballico@unitn.it}

\address{Dept. of Mathematics, Texas A\&M University,
College Station, TX 77843-3368, USA}
\email{eventura@math.tamu.edu, emanueleventura.sw@gmail.com}


\begin{document}

\maketitle

\begin{abstract}
Many classical results in algebraic geometry arise from investigating some extremal behaviors that appear
among projective varieties not lying on any hypersurface of fixed degree. We study two numerical invariants attached to such collections of varieties: their minimal degree and their maximal number of linearly independent smallest degree hypersurfaces passing through them. We show results for curves and surfaces, and pose several questions. 
\end{abstract}

\section{Introduction}\label{intro}

\noindent In this note, we study two numerical invariants attached to projective varieties, focusing on the low-dimensional cases of curves and surfaces. To introduce the problem, let $X$ be an $m$-dimensional integral projective variety in $\mathbb P^n$. Let $s$ be an integer with the property that $X$ is not contained in any hypersurface of degree strictly smaller than $s$, i.e., $h^0(\mathcal I_X(s-1))=0$.  Perhaps, the first basic question one might ask is as follows: 

\begin{question}\label{mindeg}
What is the minimal degree that such an $X$ may have?
\end{question}

Moreover, one might wonder what is the largest number of linearly independent hypersurfaces of degree $s$ passing through such an $X$, or more formally:

\begin{question}\label{maxnumb}
For such a projective variety $X$, how big $h^0(\mathcal I_X(s))$ may be?
\end{question}

Both Question \ref{mindeg} and Question \ref{maxnumb} concern extremal behaviors, which are often fundamental phenomena in algebraic geometry. When $X$ is a curve, a variant to the questions above is about the maximal genus of a curve given a prescribed linear series on it, or the maximal genus of a curve not lying on a surface of fixed degree: these are the subjects of the famous Castelnuovo's and Halphen's theories respectively, which have been of crucial importance in the theory of curves and much beyond; see, e.g., the works of Chiantini and Ciliberto \cite{cc}, Di Gennaro and Franco \cite{df}, for results on Castelnuovo-Halphen's theory in higher dimensional projective spaces. In the same vein, another line of research investigates the $k$-normality of a projective variety, i.e., the objective is to find out the least integer $k$ such that the system of degree $k$ hypersurfaces of the ambient projective space cut out a complete linear system on $X$; see, e.g., the works of Gruson, Lazarsfeld, and Peskine \cite{glp}, Lazarsfeld \cite{laz87}, Alzati and Russo \cite{ar} for developments in this direction. 

Question \ref{mindeg} may be regarded as the very first step towards classifying non-degenerate minimal degree varieties not lying on any hypersurface of degree $< s$. Therefore the classification of the usual non-degenerate minimal degree varieties may be viewed as the case when $s=2$ \cite{hsv, park1}. 
 
Park \cite[Problem A]{park} posed the problem of determining the maximal possible values of $h^0(\mathcal I_X(s))$, for every $s$ without further assumptions; thus our Question \ref{maxnumb} is more restrictive but different, inasmuch as we require the integer $s$ to be the minimal such that $h^0(\mathcal I_X(s))\neq 0$. Park's problem is in fact classical and had previously led, for instance, to significant results such as a characterization of minimal degree varieties by Castelnuovo in terms of the number of linearly independent quadrics passing through them. The answer to this problem in the case of curves is now well-known: an upper bound was first obtained by Harris \cite{harris81}, and later improved with different methods by L'vovsky \cite{Lvov96}. Furthermore, Park's extremal cases are related to the study of varieties $X\subset \PP^n$ of almost-minimal degree, i.e., $\deg(X)=\textnormal{codim}(X)+2$; see \cite{bs, fu, park1, hsv}.\\

\noindent {\bf Contributions and structure of the paper.} In \S\ref{notation}, we define the collection of projective varieties $\Aa(n,s,m)$ we are concerned with, and two basic numerical invariants: the minimal degree $d_{n,s,m}$ appearing among all the varieties in $\Aa(n,s,m)$, and 
the maximal number $\alpha(n,s,m)$ of linearly independent hypersurfaces of degree $s$ vanishing on a given element of $\Aa(n,s,m)$. 
The study of these invariants has classical roots, as pointed out in \S\ref{intro}, and yet is new, to the best of our knowledge.  

In \S\ref{curves}, we focus on the case of curves. Proposition \ref{i1} establishes the value of the minimal degree $d_{n,s,1}$, whereas Lemma \ref{i2}
gives the possible values of arithmetic genera of curves in $\Aa(n,s,1)$ with minimal degree.

In Lemma \ref{i00+}, we record the value of $h^0(\mathcal I_X(s))$. We conjecture a bound on its maximum value, $\alpha(n,s,1)$. Question \ref{qw0} asks whether this is the maximum possible. In Remark \ref{qw00}, we observe instances where the bound proposed in Question \ref{qw0} does hold. 
The proof of Proposition \ref{i1} shows that the minimal degree is reached by some smooth rational curves: Question \ref{xq1} asks whether this is possible for any other degree $d > d_{n,s,1}$. Remark \ref{numrangexq1true} points out a numerical range where the latter question has an affirmative answer. 
In Remark \ref{x2}, we collect our current knowledge around Question \ref{xq1} in the case of $n=3$. Here we generalize {\bf range A} of curves
in $\PP^3$ to curves embedded in higher projective spaces and pose several questions about them (Question \ref{questiongeneralizedrangeA}). 
Finally, in Proposition \ref{ccc1}, using curves in $\Aa(n,s,1)$, we produce a non-degenerate irreducible projective variety (that is not a cone) of arbitrarily high degree in $\Aa(n,m,s)$, for every $m\geq 1$, with $n\geq m+3$. 

In \S\ref{surfaces}, we deal with the case of surfaces. In Theorem \ref{ma1}, we show useful upper bounds for the dimension of linear systems on a smooth surface, using one of the building blocks of Mori theory (the {\it Kawamata Rationality Theorem}). Remark \ref{ma2} and Proposition \ref{be1} (proven in \cite{bea}) point out some circumstances where we can derive some information on $d_{n,s,2}$ and $d_{n,3,2}$, respectively. The findings in Theorem \ref{ma1}, along with Proposition \ref{be1}, are summarized in Theorem \ref{be2}. Finally, Remark \ref{w3} collects some cohomological facts about almost-minimal degree surfaces. \\
\vspace{3mm}

\begin{small}
\noindent {\bf Acknowledgements.} The first author was partially supported by MIUR and GNSAGA of INdAM (Italy). The second author would like to thank the Department of Mathematics of Universit\`{a} di Trento, where part of this project was conducted, for the warm hospitality. 
\end{small}

\section{Two numerical invariants}\label{notation}
Our varieties are over the complex numbers. We introduce the collections of projective varieties we are concerned with: 

\begin{definition}
Let $m \ge 1$, $s\ge 3$ and $n\ge m+2$. Define 
\[
\Aa (n,s,m) = \left\lbrace \mbox{integral $m$-dimensional } X\subset \PP^n \ | \   h^0(\Ii _X(s-1)) =0 \right\rbrace. 
\] 
\end{definition}

Since $s>1$, any $X\in \Aa (n,s,m)$ is non-degenerate, i.e., it spans the ambient projective space, as one clearly has $h^0(\Ii_X(1)) = 0$. 

\begin{definition}
We introduce two numerical invariants: 
\[
d_{n,s,m} = \min \left\lbrace \deg(X) \ | \  X\in \Aa (n,s,m) \right\rbrace, \mbox{ and }
\]  
\[
\alpha (n,s,m) = \max \left\lbrace h^0(\Ii _X(s)) \ | \ X\in \Aa(n,s,m) \right\rbrace. 
\] 

Hence $\alpha(n,s,m)$ is the maximal number, over the set $\Aa(n,s,m)$, of the linearly independent hypersurfaces of degree $s$ vanishing on some $X\in \Aa(n,s,m)$.\end{definition}

We conclude this preliminary section with an observation about minimal degrees: 
\begin{remark}
For $m\ge 2$, we have $d_{n,s,m} \le d_{n-1,s, m-1}$. To see this, take $X\subset \mathbb P^{n-1}$ of dimension $m-1$ with minimal degree $\deg(X)=d_{n-1,s,m-1}$. Let $C(X)$ be the cone over $X$ with a single point as vertex: this sits in $\PP^n$, it has dimension $m$ and degree $d_{n-1,s,m-1}$. Moreover, it is clear that $C(X)\in \Aa(n,s,m)$. 
\end{remark}
\section{Curves}\label{curves}

In this section, we focus on the case of curves, i.e., $m=1$. For a curve $X$, $p_a(X)$ denotes its arithmetic genus. 
For the ease of notation, we set 
\[
\Aa (n,s):= \Aa (n,s,1), d_{n,s}:= d_{n,s,1}, \mbox{ and } \alpha (n,s):= \alpha (n,s,1).
\]
We introduce another piece of notation: 
\[
H(d,g;n) = \left\lbrace \mbox{smooth, connected, non-deg. } X\subset \PP^n, \deg(X) = d, p_a(X) = g\right\rbrace. 
\] 
In the range of degree $d\ge g+1$, define 
\[
H(d,g;n)' = \left\lbrace X\in H(d,g;n)\ | \ h^1(\Oo _X(1)) =0\right\rbrace. 
\]
Note that $H(d,g; n)'=H(d,g; n)$ whenever $d>2g-2$, as $\Oo_X(1)$ is non-special. 

\begin{proposition}\label{i1}
The numerical invariant $d_{n,s}$ satisfies
\[
d_{n,s} = \left\lceil \left(\binom{n+s-1}{n}-1\right)/(s-1)\right\rceil.
\]
\end{proposition}
\begin{proof}
Let 
\[
\delta =  \left\lceil \left(\binom{n+s-1}{n}-1\right)/(s-1)\right\rceil.
\]
We show that $d_{n,s} = \delta$. 

For any integer $d\ge n$, the set $H(d,0;n)$ of all smooth and non-degenerate degree $d$ rational curves $X\subset \PP^n$ has the structure of a smooth and integral algebraic variety, whose general point $X$ has {\it maximal rank}: hence $h^0(\Ii _X(t)) =0$ for some $t\in \NN$ if and only if $d t+1\geq h^0(\Oo_{\PP^n}(t))$, i.e., whenever it is possible for the restriction morphism to be injective; see \cite{hi, be3}. Thus a general $X\in H(d,0;n)$ satisfies
\[
h^0(\Ii _X(s-1)) =0 \ \ \mbox{ if and only if } \ \ (s-1)d+1\ge \binom{n+s-1}{n}. 
\] 
This implies that $d_{n,s}\leq \delta$, as we have a (rational) integral curve of degree $\delta$. 

For the converse inequality, consider any integral curve $Y\subset \PP^n$ such that $h^0(\Ii _Y(s-1)) =0$, and let $k = \deg (Y)$. Since $h^0(\Ii _Y(s-1)) =0$, we have $h^0(\Oo _Y(s-1)) \ge \binom{n+s-1}{n}$. Since $Y$ is an integral curve (of arbitrary arithmetic genus), elementary considerations give 
\[
h^0(L) \le \deg (L)+1,
\] 
for an arbitrary line bundle $L$ on $Y$. Thus 
\[
\deg(\Oo_Y(s-1))+1 = (s-1)k+1 \ge h^0(\Oo_Y(s-1))\geq \binom{n+s-1}{n}.
\]
Taking the minimum, the latter inequality implies $d_{n,s}\geq \delta$, which completes the proof. 
\end{proof}

\begin{lemma}\label{i2}
Let $X\subset \PP^n$ be an integral curve $X\in \Aa(n,s)$ such that $\deg (X)=d_{n,s}$. Then its arithmetic genus satisfies
\[
p_a(X) \le (s-1)d_{n,s}+1-\binom{n+s-1}{n}.
\]
\end{lemma}

\begin{proof}
Since $h^0(\Ii _X(s-1))=0$, Riemann-Roch shows that the statement is true if $h^1(\Oo _X(s-1)) =0$. Otherwise, assume $h^1(\Oo _X(s-1)) >0$. Since $h^0(\Oo _X(s-1)) \ge \binom{n+s-1}{n}$, Clifford's theorem gives $(s-1)d_{n,s} \ge 2\binom{n+s-1}{n}-2$. By Proposition \ref{i1}, we have 
\[
(s-1)(d_{n,s} -1)+2 \le \binom{n+s-1}{n}.
\]
Altogether, they give $\binom{n+s-1}{n}\le s-1$, contradicting the assumption $s\ge 3$.  
\end{proof}

\begin{remark}\label{i3}
Fix an integer $g$ such that $0\le g \le  (s-1)d_{n,s}+1-\binom{n+s-1}{n}$. Again, as in the proof of Lemma \ref{i2}, Proposition \ref{i1} yields
\[
(s-1)(d_{n,s} -1)+2 \le \binom{n+s-1}{n},
\] 
and so $g\le s-2$. 
\vspace{3mm}

\begin{quote}
{\it Claim.} We have $s-2 \le d_{n,s}-n$.\\
{\it Proof of the Claim.} By contradiction, assume $d_{n,s} \le n+s-3$. By Proposition \ref{i1}, in order to obtain a contradiction, it is enough to show 
\[
(s-1)(n+s-3) +2 \le \binom{n+s-1}{n}.
\] 
Let $\phi (n,s) = \binom{n+s-1}{n} - (s-1)(n+s-3)-2$. This function is non-negative in the desired range. Indeed, consider
\[
\phi (3,s) = (s+2)(s+1)s/6 -s(s-1)-2.
\]
It is clear this is non-negative for all $s\ge 3$. Moreover: 
\[
\phi (n+1,s) -\phi (n,s) =\binom{n+s-2}{n} -s+1 \ge 0, \ \ \mbox{ for all } \ \  s\ge 3, 
\]
which leads to a contradiction. 
\end{quote}

By the {\it Claim} and \cite{be2, be3}, a general curve $X\in H(d_{n,s},g;n)'$ has maximal rank. Moreover, for each integer $g$ with $0\le
g\le  (s-1)d_{n,s}+1-\binom{n+s-1}{n}$, there exists a curve $X\in H(d_{n,s},g;n)$ such that $h^0(\Ii _X(s-1))=0$. Since $h^1(\Oo _X(1)) =0$, we
have $h^0(\omega_X(-1)) = 0$. The latter condition is equivalent to $h^0(\omega_X(-t))=0$ for all $t\geq 1$. 
Hence $h^1(\Oo _X(t)) =0$ for $t=s-1,s$. Moreover, we have: 
\[
h^1(\Ii _X(s-1)) = (s-1)d_{n,s}+1-g -\binom{n+s-1}{n},
\]
\[
h^0(\Ii _X(s)) = h^1(\Ii _X(s)) +\binom{n+s}{n} -sd_{n,s}-1+g.
\] 
Now, choose $g_0 = (s-1)d_{n,s} + 1 -\binom{n+s-1}{n}$, so that $h^1(\Ii _X(s-1)) =0$. 
Furthermore 
\[
h^2(\Ii _X(s-2)) =h^1(\Oo _X(s-2)) =0,
\] 
where the right-most equality holds because $s\ge3$. By the well-known Castelnuovo-Mumford's Lemma, we derive 
\[h^1(\Ii _X(t)) =0 \mbox{ for all } t\ge s. \] 
Finally,
\[
h^0(\Ii _X(s)) =\binom{n+s}{n} -sd_{n,s}-1+g_0 = \binom{n+s-1}{n-1}-d_{n,s}.
\]
\end{remark}

To summarize, in Remark \ref{i3}, we have shown the following: 

\begin{remark}\label{observation}
Setting $g_0 = (s-1)d_{n,s} + 1 -\binom{n+s-1}{n}$, We have: 
\[
d_{n,s} \ge n+g_0;
\]
additionally, a general curve $X\in H(d,g_0;n)'$ satisfies 
\[
h^0(\Ii _X(s-1)) =0 \mbox{ and } h^0(\Ii _X(s))=\binom{n+s-1}{n-1} -d_{n,s}.
\]
\end{remark}

\begin{question}\label{qw0}
Is it true that 
\[
\alpha (n,s) = \binom{n+s-1}{n-1} -d_{n,s}?
\] 
Keep the definition of $g_0$ from Remark \ref{observation}. Does equality hold if and only if $X$ is a general curve in $H(d_{n,s}, g_0; n)'$?
\end{question}

\begin{lemma}\label{i00+}
Let $n\ge 3$ and $s\ge3$. Let $X\subset \PP^n$ be an integral and non-degenerate curve such that $h^0(\Ii _X(s-1))  = 0$. Fix a general
hyperplane $H\subset \PP^n$.
Then 
\begin{align}\label{eqi00+} 
&
h^0(\Ii _X(s))  =  \binom{n+s-1}{n-1}-\deg (X)  -h^1(\Ii_X(s-1)) + h^1(\Ii _X(s))+ \\
& \nonumber
 \ \ \ \ \ \ \ \ \ \ \ \ \ \ \ \ \ \ \ \  \ \ \ \ + h^1(\Oo _X(s-1)) -h^1(\Oo _X(s)).
\end{align}
\end{lemma}

\begin{proof}
Fix a general hyperplane $H\subset \PP^n$. Note that $h^0(H,\Ii _{X\cap H}(s)) =\binom{n+s-1}{n-1}-\deg (X)+h^1(H,\Ii _{X\cap H}(s))$. Since $h^1(\Oo _X(t)) =h^2(\Ii _X(t))$ for all $t\in \mathbb Z$, in order to conclude it is sufficient to use the long exact sequence in cohomology of the exact sequence
\begin{equation}\label{eqi01}
0\to \Ii _X(s-1)\to \Ii _X(s)\to \Ii _{X\cap H,H}(s)\to 0, 
\end{equation}
along with the assumption $h^0(\Ii _X(s-1)) =0$.
\end{proof}

Now we discuss the term appearing on the right-hand side of equality (\ref{eqi00+}). 

\begin{remark}\label{qw00}
In the setting of Lemma \ref{i00+}, as usual let $d = \deg (X)$ and $g = p_a(X)$. 
Since $h^0(\Ii _X(s-1)) =0$, by Riemann-Roch, we have 
\[
h^1(\Ii _X(s-1)) = (s-1)d +1-g-\binom{n+s-1}{n}.
\] 
Define the numerical function $\psi (t) := h^1(\Oo _X(t-1)) -h^1(\Oo _X(t))$. For fixed integers $g$, $h^1(\Ii _X(s))$ and $\psi(s)$, the right-hand side of (\ref{eqi00+}) is a strictly decreasing function in $d$.

Since $\deg (\Oo _X(t)) =d+\deg (\Oo _X(t-1))$, Serre duality and Riemann-Roch give 
\[
0\le \psi (t) \le d.
\]
Therefore a class of curves for which the term $h^1(\Oo _X(s-1)) -h^1(\Oo _X(s))$ is zero is realized when $h^1(\Oo _X(s-1)) =0$.

Moreover, note that the long exact sequence in cohomology of (\ref{eqi01}) yields
\[
h^0(\Ii _X(s)) \le \binom{n+s-1}{n-1} -d,
\] 
whenever $h^1(H,\Ii _{X\cap H,H}(s)) =0$. As in \cite[Chapter 3]{szp}, let $\chi$ denote the minimal non-negative integer such that $h^1(H,\Ii _{X\cap H,H}(\chi +1)) =0$. As classically shown by Castelnuovo, using (\ref{eqi01}) one obtains $h^1(\Oo _X(t)) =0$ for all $t\ge \chi$ \cite[Theorem 1, p. 52]{szp}.
Furthermore, $h^1(\Ii _X(t)) \ge h^1(\Ii _X(t+1))$ for all $t\ge \chi$ \cite[Lemma 2, p. 53]{szp}, 
and $h^1(\Ii _X(t))> h^1(\Ii _X(t+1))$ whenever $t>\chi$ and $h^1(\Ii _X(t))\ne 0$.

\end{remark}

Note that the proof of Proposition \ref{i1} shows that a curve in $\mathcal A(n,s)$ of minimal degree $d_{n,s}$ can be chosen to be smooth (and rational). 
More generally, we wonder the following: 

\begin{question}\label{xq1}
Let $n\ge 3$, $s\ge 3$ and $d> d_{n,s}$. 
\begin{enumerate}
\item[(i)] Is there an integral and non-degenerate curve $X\subset \PP^n$ such that $\deg (X)=d$, $h^0(\Ii _X(s-1)) =0$ and
$h^0(\Ii _X(s))\ne 0$?
\item[(ii)] Can we choose such an $X$ to be smooth?
\end{enumerate}
\end{question}

In the next two remarks we recall our knowledge around Question \ref{xq1} for $n=3$. 

\begin{remark}\label{x2}
Recall that, by Proposition \ref{i1}, 
\[
d_{3,s} = \left\lceil \left(\binom{s+2}{3}-1\right)/(s-1)\right\rceil.
\]  
For $d, s\ge 3$, to  each pair $(d,s)$, we attach the set of degree $d$ curves in $\PP^3$ 
{\it not} contained in any surface of degree $<s$. Given $(d,s)$, Halphen's question asks in particular to determine the maximum genus of a curve associated to this pair. It is customary to divide the set of pairs $(d,s)$ into four regions: \\

{\bf Range $\emptyset$}: $d < (s^2+4s+6)/6 $. In this range, there is no integral and non-degenerate curve $X\subset \PP^3$ such that
$\deg (X)=d$ and $h^0(\Ii _X(s-1))=0$ \cite[Theorem 3.3]{h0}, \cite{hh}. 

\vspace{5mm}
{\bf Range A}: $(s^2+4s+6)/6 \le d < (s^2+4s+6)/3$. In this range, one finds curves with  
$h^1(\Oo _X(s-1)) =0$ and hence $h^1(\Oo _X(s))=0$. Thus if $h^0(\Ii _X(s-1))=0$ and $h^0(\Ii _X(s)) \ne 0$, by Riemann-Roch one has $(s-1)d+1-g
\ge \binom{s+2}{3}$ and $sd+1-g\le \binom{s+3}{3}$. Moreover, if $\lceil s(s+2)/4\rceil \le d <(s^2+4s+6)/3$, it is known such a curve exists for any such choice of $(d,s)$, along with curves reaching the maximal genera allowed; see, e.g., \cite{bbem, fl1,fl2}.  To fill in all 
$(d,s)$ such that $(s^2+4s+6)/6\le d < \lceil s(s+2)/4\rceil$ the same result (the existence of curves with $(d,s)$ and $g= (s-1)d+1- \binom{s+2}{3}$) is only known for $s$ large enough, i.e., $s \geq 10.5\times 10^5$, by the recent \cite[Theorem 1]{be6}. 

\vspace{5mm}
{\bf Range B}: $(s^2+4s+6)/3\le d < s(s-1)$. Hartshorne and Hirschowitz constructed, for any $(d, s)$ in this range, smooth space curves with very high genera and conjectured that they have the maximal genus for all curves associated to $(d,s)$ \cite{hh, ms}. 

\vspace{5mm}
{\bf Range C}: $d>s(s-1)$. One finds a smooth degree $d$ space curve $X$ such that
$\deg (X)=d$ and $h^0(\Ii _X(s-1))=0$; Gruson and Peskine classified all such curves $X$ of maximal genus \cite{gp1}: 
if $d\equiv 0\pmod{s}$, these curves are complete intersections of a surface of degree $s$ and a surface of degree $d/s$; in all other
cases $X$ is {\it linked} to a plane curve $C$ of degree $\deg(C) = s\lceil d/s\rceil -d$, by the complete intersection of a surface of degree $s$
and a surface of degree $\lceil d/s\rceil$. For $d\geq s^2$, we have $h^0(\Ii _X(s)) \le 2$ with equality if and only if $d=s^2$
and $X$ is the complete intersection of two surfaces of degree $s$.
\end{remark}

\begin{remark}\label{x3}
Assume $(d,s)$ is in {\bf Range A}. Let
$X\subset \PP^3$ be any integral curve with $\deg (X)=d$, $h^1(\Oo _X(s-1)) =0$, $h^0(\Ii _X(s-1)) =0$, and $h^0(\Ii _X(s))
\ne 0$, i.e., by Riemann-Roch one has $(s-1)d+1-g\ge \binom{s+2}{3}$, and $sd+1-g\le \binom{s+3}{3}$, where $g = p_a(X)$ is the arithmetic genus
of $X$. All curves constructed in \cite{bbem, be6, bef, fl1, fl2} have $h^1(\Ii _X(s-1)) = h^2(\Ii _X(s-1)) =0$ and hence $h^0(\Ii _X(s)) =\binom{s+2}{2} - d$. The ones in the main results of \cite{bbem, be6} have $g= d(s-1)+1-\binom{s+2}{3}$, whereas the ones given in \cite[Proposition 4.3]{bbem} and in \cite{bef}
have many different genera for the fixed pair $(d,s)$.
\end{remark}

We generalize {\bf Range A} to curves embedded in $\PP^n$ for any $n\geq 3$. For all integers $n\geq 3$, $s\ge 2$ and $d\ge d_{n,s}$, let $\Bb (n,s,d)$ (resp. $\Bb (n,s,d)'$) denote the set of all smooth and connected (resp. integral) $X\in \Aa (n,s)$ such that $\deg (X)=d$ and $h^1(\Oo _X(s-1)) = 0$.

\begin{definition}
Curves in $\Bb (n,s,d)$ are said to be in the {\it generalized} {\bf Range A}.
\end{definition}

\begin{remark}
Let $X\in \Bb (n,s,d)'$. Since $h^1(\Oo _X(s-1))  =0$, 
Riemann-Roch gives $h^0(\Oo _X(s-1)) =d(s-1)+1-p_a(X)$. Since $h^0(\Ii _X(s-1))=0$, one obtains $h^0(\Oo _X(s-1))\geq h^0(\Oo
_{\PP^n}(s-1)) =\binom{n+s-1}{n}$, and so: 
\[
p_a(X)\le d(s-1)+1-\binom{n+s-1}{n}.
\]
\end{remark}

\begin{question}\label{questiongeneralizedrangeA}
We pose the following questions:
\begin{enumerate}
\item[(i)] For which choices of $d, s, n$, does there exist $X\in \Bb (n,s,d)$ (or even $X\in \Bb (n,s,d)'$) such that $p_a(X) =d(s-1)+1-\binom{n+s-1}{n}$?
\item[(ii)] What is the maximal arithmetic genus of curves in $\Bb (n,s,d)$ (or in $\Bb (n,s,d)'$)?
\item[(iii)] For which $g_0$ and every $0\le g \le g_0$, does there exist $X\in \Bb (n,s,d)$ (or $X\in \Bb (n,s,d)'$) with such arithmetic genus? 
\end{enumerate}
\end{question}

\begin{remark}
For any $X\in \Bb (n,s,d)'$, one has $h^2(\Ii _X(s-1))=0$. Suppose there exists $X\in \Bb (n,s,d)'$ such that $p_a(X) =d(s-1)+1-\binom{n+s-1}{n}$. Thus 
$h^0(\Ii _X(s-1)) = h^2(\Ii _X(s-1)) =0$. By Riemann-Roch, $h^0(\Oo _X(s-1))=0$, and hence $h^1(\Ii _X(s-1)) =0$. Therefore, these curves are exactly the integral and non-degenerate degree $d$
curves $X\subset \PP^n$ such that $h^i(\Ii _X(s-1)) =0$ for all $i\ge 0$; this is because for any curve $Y\subset \PP^n$, one has
 $h^i(\Ii_Y(t)) =h^{i+1}(\Oo _{\PP^n}(t))=0$ for all $i\ge 3, i\neq n$ and all $t\in \ZZ$, and for all $t>-n$ when $i=n$. 
 
Take any hyperplane $H\subset \PP^n$. By the exact sequence (\ref{eqi01}), we
have $h^1(\Ii _X(s)) = h^1(H,\Ii _{X\cap H,H}(s))$. 
Assume $h^1(\Oo _X(s-2))=0$, i.e., $h^1(\Ii _X(s-2))=0$.
Since $h^i(\Ii _X(t)) =0$ for all $i\ge 3$ and $t> -n$, we have $h^i(\Ii _X(s-1-i)) =0$ for all $i\ge 0$. In this case,
Castelnuovo-Mumford's lemma gives $h^1(\Ii _X(t)) =0$ for all $t\ge s$: this implies that the homogeneous ideal of $X$ is generated by
$\textnormal{H}^0(\mathcal I_X(s))$. 

Now we drop the assumption $h^1(\Oo _X(s-2))=0$, and instead suppose $h^1(\Ii _X(s)) =0$,
i.e., $h^1(H,\Ii _{X\cap H,H}(s))=0$. This implies $d\le \binom{n+s-1}{n-1}$. In this case,
the Castelnuovo-Mumford's lemma gives $h^1(\Ii _X(t)) =0$ for all $t>s$, which implies that the homogeneous ideal of $X$ is generated by
$\textnormal{H}^0(\mathcal I_X(s))$ and $\textnormal{H}^0(\mathcal I_X(s+1))$.
\end{remark}

\begin{remark}\label{numrangexq1true}
Fix integers $n, d, g, s$ such that $n\ge 3$, $s\ge 3$, with 
\begin{enumerate}
\item[(i)] $0 \le g \le d-n$, 

\item[(ii)] $(n+1)d \ge ng +n(n+1)$,

\item[(iii)] $(s-1)d \ge \binom{n+s-1}{n} +g-1$, and

\item[(iv)] $sd \le \binom{n+s-1}{n}+g-2$.
\end{enumerate}
By results shown in \cite{be2, be4, be5}, there exists a smooth, connected, and non-degenerate {\it maximal rank} 
curve $X\subset \PP^n$ with $\deg (X) = d$, arithmetic genus $p_a(X)=g$, and $h^0(\Ii _X(s)) =\binom{n+s}{n} -sd+g-1$. Hence, setting the arithmetic genus to be the maximal in this range, i.e., $g=d-n$, (i) and (ii) in Question \ref{xq1} have a positive answer, whenever
\[
h^0(\Ii _X(s))>0 \Longleftrightarrow (s-1)d < \binom{n+s}{n} -n-1.
\]
Note that inequality (iii) implies $h^0(\Ii _X(s-1)) = 0$, because $X$ has maximal rank. 
\end{remark}

Using curves of minimal degree in $\Aa(n,s)$, for every $m\ge 1$ (with $n\geq m+3$), we construct an non-degenerate, irreducible variety (that is not a cone) of arbitrarily high degree in $\Aa(n,s,m)$; this is accomplished in Proposition \ref{ccc1}, based on the vector bundle construction featured in the next remark. 

\begin{remark}\label{ccc0}
Let $m \ge 1$ and $e\ge 2$. Consider the rank $m+1$ vector bundle $\Ee = \Oo _{\PP^1}\oplus \Oo
_{\PP^1}(-e)^{\oplus m}$ on $\PP^1$. Define the corresponding projective bundle $T = \PP(\Ee)$ and let $\pi: T\to \PP^1$ be the vector bundle projection on $\PP^1$. The map $\pi$ makes $T$ a $\PP^{m}$-bundle and hence $\mathrm{Pic}(T) \cong \ZZ^2$, with basis given by a fiber $f$ of $\pi$ and any line bundle
on $T$ whose restriction to any fiber of $\pi$ has degree one; see, e.g., \cite[Ex. II.7.9]{h}. Among these
generators we take the only one, say $h$, such that
$|h| =\{h\}$ (this corresponds to the unique surjection $\Oo _{\PP^1}\oplus \Oo
_{\PP^1}(-e)^{\oplus m} \to \Oo _{\PP^1}$, \cite[Ch. II]{h}). As an abstract variety,
$h$ is the trivial $\PP^{m-1}$-bundle over $\PP^1$ and hence $h\cong \PP^1\times \PP^{m-1}$. 

Similarly to the case of surfaces \cite[Ch. V]{h}, for any $a, b\in\ZZ$ we have $|ah+bf|\ne\emptyset$ if and only if either $a=0$ and $b\ge 0$ or $a>0$ and $b\ge ae$. The line bundle $\Oo_T(ah+bf)$ is globally generated (resp. ample) if and only if $a\ge 0$ and $b\ge ae$ (resp. $a>0$ and $b>ae$). 

The complete linear system $h+ef$ induces a morphism $u: T\to \PP^n$, where $n =  e+m$, as $h^0(\Oo _{\PP^1}(e)\oplus \Oo _{\PP^1}^{\oplus m}) 
=e+1+m$. Let $W= u(T)$ be the image of $T$. The morphism $u$ contracts $h$ to an $(m-1)$-dimensional linear space in $\PP^r$, whereas $u|_{T\setminus h}$ is an
embedding. Thus $u$ is a morphism birational onto its image, as the support of $h$ is codimension-one. \\

\begin{quote}
{\it Claim.} The variety $W$ has degree $e$. \\
{\it Proof of the Claim.} Since $u$ is a morphism birational onto its image, $\deg (W)= (h+ef)^{m+1}$, where the latter
integer is the intersection product in the Chow ring of $T$ of $m+1$ copies of the divisor $h+ef$ defining the morphism. 
Since any two different fibers of $\pi$ are disjoint, the class $f^2$ in the Chow ring is zero. Thus $(h+ef)^{m+1} = h^{m+1} +(m+1)eh^mf$.
Fix a fiber $F\in |f|$ of $\pi$. Since $h_{|F}$ is the degree one line bundle on $F$ and $F\cong \PP^{m}$, we have $h^{m}f =1$ and
hence $(m+1)eh^{m}f =(m+1)e$. \\
If $m=1$, $T$ is a rational ruled surface. The normal bundle of $h$ has degree $h^2=-e$. Hence the conclusion follows for $m=1$. \\
Assume $m\ge 2$. Since $h \cong \PP^1\times \PP^{m-1}$, we have
$\mathrm{Pic}(h)\cong \ZZ^2$. As generators of the lattice $\mathrm{Pic}(h)$, we take the pullbacks of hyperplane classes of the factors through the projections $\pi _1: h\to \PP^1$ and $\pi _2: h\to \PP^{m-1}$, i.e., $\Oo _h(1,0) = \pi _1^\ast (\Oo_{\PP^1}(1))$ and $\Oo _h(0,1) = \pi _2^\ast (\Oo _{\PP^{m-1}}(1))$. \\
By definition, the restriction $h_{|h}$ is the normal bundle of $h$ in $T$ and hence $h_{|h} \cong
\Oo _h(-e,1)$, by induction on $m$ and using standard exact sequences on normal bundles. Thus $h^{m+1} = h^m|_{h} = -me$, concluding the proof.
\end{quote}

By this {\it Claim}, $W$ is a minimal degree $(m+1)$-dimensional subvariety of $\PP^{n}$, where $n=e+m$. It is a cone whose vertex is the $(m-1)$-dimensional linear space $V= u(h)$, and the base of the cone is a rational normal curve in any linear space $M\subset \PP^n$ such that $\dim M = n-m = e$ and $M\cap V=\emptyset$. Thus $W$ is arithmetically Cohen-Macaulay. In particular, for any integer $s>0$, a divisor $X\subset W$ is contained in a degree $s$ hypersurface of $\PP^n$ if and only if its pullback $X'$ to $T$ is a part of the linear system $|sh+sef|$, i.e., there exists a divisor $X''\ge 0$ such that $X'+X'' \in |sh+sef|$. Thus $X$ is not contained in any degree
$s$ hypersurface of $\PP^n$ if $X'\in |h+bf|$ for some integer $b \geq se+1$. In the Chow ring, we have:
\[\deg (X) =X'\cdot (h+ef)^m = (h+bf)\cdot(h+ef)^m = 
\]
\[
=h^{m+1}+meh^{m}f+b = b.
\]
\end{remark}

We are now ready to prove the following

\begin{proposition}\label{ccc1}
Let $m\ge 1$, $n\ge m+3$, $s\ge 3$ and $d\ge (s-1)d_{n-m,s}+1$. There exists an integral and non-degenerate
$m$-dimensional variety $X\subset \PP^n$, with $\deg (X)=d$ and $h^0(\Ii _X(s-1)) =0$. Moreover, $X$ is not
a cone.
\end{proposition}

\begin{proof}
Let $H\subset \PP^n$ be a linear subspace with $\dim H = n-m$. Take a linear subspace $E\subset \PP^n$ such that $\dim E =m-1$
and $E\cap H=\emptyset$. By \cite{hi, be2} (for the case $n=4$), \cite{be2} (for the case $n=5$) and \cite{be3} (for the case $n\ge 6$), there exists a smooth rational curve $Y\subset H$ such that 
\[
\deg (Y) =d_{n-m,s},\ \ h^0(H,\Ii _{Y,H}(s-1)) =0,\ \mbox{and} \ \ h^0(H,\Ii _{Y,H}(s)) = \binom{n+s-m}{n-m} -d_{n-m,s}.
\]

Let $W'$ be the cone with vertex $E$ and such a curve $Y\subset H$ as its base. Up to a projective transformation, the
projective variety depends only on $Y$, and not on the choice of a linear space $E$ such that $E\cap H = \emptyset$. 
Consider the setting described in Remark \ref{ccc0} with $e =d_{n-m,s}$. Note that $W'$ is projectively equivalent to the $(m+1)$-dimensional minimal degree variety $W$ featured in Remark \ref{ccc0}. Let $X\subset W'$ be an integral divisor such that $X\in |h+df|$, upon an identification $W\cong W'$. Since $d\geq (s-1)d_{n-m,s}+1$, by Remark \ref{ccc0}, $X$ is not contained in any hypersurface of degree $s-1$. 

If $X$ were a cone, either its vertex is contained in $E$ or is a point on $Y$. In the first case, upon identification with $W$, as in Remark \ref{ccc0}, let $X'=u^{*}X$ be the pullback of $X$ through $u$. The image of $h\cap u^{*}X$ under the morphism $u$ would be a codimension-one linear space in $u(h)\cong E$. However, this is impossible, as the degree of $u(h\cap X)$ is $d-d_{n-m,s}>1$. 
The second case is not possible either, because the pullback $X'$ would contain the divisor $h$. However, $X'$ is integral because $u$ is an isomorphism outside $h$, and $X'$ cannot coincide with $h$, as $X$ is a divisor of $W$. 
\end{proof}

\section{Surfaces}\label{surfaces}

In this section, we shift gears to surfaces, i.e., $m=2$. We start by recording an upper bound on the dimension of a linear system on a smooth surface, derived from the Kawamata Rationality Theorem. (For a given projective surface $X$, denote by $\kappa(X)$ its Kodaira dimension.)

\begin{theorem}\label{ma1}
Let $n\ge 4$, $s\ge 2$ and $d\ge n-1$. Let $X\subset \PP^n$ be a  smooth, connected and non-degenerate degree $d$ surface. Then:
\begin{enumerate}

\item[(i)] If $\kappa (X)\ne -\infty$, then $h^0(\Oo _X(s)) \le 2+s^2d/2$. 

\item[(ii)] If $\kappa(X) = -\infty$, $X\ncong \PP^2$ and $X$ is $\PP^1$-bundle over a smooth curve $\pi: X\rightarrow D$ such that the rational fibers have degree one (i.e., $X$ is a scroll), then $h^0(\Oo _X(s)) \le 1 + (s^2+2s)d/2$. 

\item[(iii)] We have $h^0(\Oo _X(s)) \le 1 + (s^2+3s)d/2$ and equality holds if and only if $d=k^2$ for some $k$ such that
$\binom{k+2}{2}\ge n+1$, $X\cong \PP^2$ and $X$ is isomorphic to either the degree $k$ Veronese embedding of $\PP^2$ or to
an isomorphic linear projection of it. 

\item[(iv)] Otherwise, if $\omega_X^{\otimes 2}(3)$ is nef, then $h^0(\Oo _X(s)) \le 1 + (s^2+\frac{3}{2}s)d/2$. 
\end{enumerate}
\end{theorem}

\begin{proof}
Fix a general $C\in |\Oo _X(s)|$. By Bertini's theorem, $C$ is a smooth and connected curve of degree $sd$. By the genus formula, $C$ has genus 
\[
g = 1 + (s^2d + \omega _X\cdot \Oo _X(s))/2. 
\]
Multiplying by the section $C$ yields the exact sequence 
\begin{equation}\label{eqma1}
0 \to \Oo _X\to \Oo _X(s)\to \Oo _C(s)\to 0.
\end{equation}

Thus $h^0(\Oo _X(s)) \le h^0(\Oo _C(s)) +1$, where equality holds if $h^1(\Oo _X)=0$. \\
If $s^2d = \deg (\Oo _C(s)) \le 2g-2$, then this divisor on $C$ is special and Clifford's theorem \cite[pp. 107--108]{acgh} gives 
\[
h^0(\Oo _C(s)) \le 1 +s^2d/2. 
\]
Whence $h^0(\Oo _X(s))\le 2 +s^2d/2$.

From now on, assume $s^2d\ge 2g-1$ (or equivalently, $\Oo _C(s)$ is non-special). Again, by the genus formula, one has 
\[
\omega _X \cdot \Oo _X(1) <0.
\] 
Note that Riemann-Roch gives $h^0(\Oo _C(s)) =s^2d +1-g = \frac{1}{2}(s^2d -s\omega _X\cdot \Oo _X(1))$.

Before we proceed, recall that a line bundle $\mathcal L$ on a projective variety $X$ of dimension $\ge 2$ is said to be \emph{nef} if $\mathcal L \cdot C\ge 0$, for every algebraic curve $C\subset X$. 
\begin{quote}
{\it Claim.} If $\omega _X \cdot \Oo _X(1) <0$, then $\kappa (X)=-\infty$. \\
{\it Proof of the Claim.} Assume on the contrary that $\kappa (X)\ge 0$. Let $X'$ be the minimal model of $X$ and $\phi: X\to X'$ be the corresponding morphism. Since $\kappa (X') =\kappa (X)\ge 0$, $\omega _{X'}$ is nef. We have $\omega _X \cong \phi^\ast\omega _{X'} +E$, where $E$ is an effective divisor and $E = 0$ if and only if $X$ is minimal. Since $\phi^\ast\omega _{X'}$ is nef and $\Oo_X(1)$ is effective, we derive 
\[
( \phi^\ast \omega _{X'} +E)\cdot \Oo _X(1)\ge 0,
\]
which is a contradiction. 
\end{quote}
The {\it Claim} shows that the Kodaira dimension of $X$ is $\kappa (X) = -\infty$, and statement (i) is proven. 

We recall the Kawamata Rationality Theorem in the case of smooth surfaces. Let $X$ be a smooth projective surface such that $\omega _X$ is {\it not} nef. For each ample line bundle $\Ll$ on $X$, the {\it nef-value} $\tau$ of $\Ll$ with respect to $\omega_X$ is defined as 
\[
\tau(\Ll) = \sup \lbrace t>0 \ | \  \Ll +t\omega _X \mbox{ is nef} \rbrace, 
\]
where  $\Ll +t\omega _X$ is viewed as a functional on the cone of curves depending on a parameter $t$. The aforementioned theorem states that the real number $\tau$ is rational, $\tau = u/v$, where $u$ is a positive integer and $v\in \{1,2,3\}$; see, e.g., \cite[Theorem 1.5.2]{bs}, \cite[Corollary 1-2-15]{matsuki}. Moreover, $\Ll +\tau \omega_X$ is nef \cite[Lemma 1.5.5]{bs}, \cite[Theorem 1-2-14]{matsuki}. Thus
we have the following cases:

\begin{itemize}
\item $u=1$ and $v=3$. The cone of curves of $X$ has an {\it extremal ray} of {\it length} $3 =\dim X +1$; see \cite{matsuki} for these notions. This implies $X\cong \PP^2$ and $\Oo _X(1)$ is the generator of $\mathrm{Pic}(X)$. Since $n\ge 4$, $X$ is embedded as described in the statement. 
\vspace{2mm}
\item $u=1$ and $v=2$. Since $\Ll+\frac{1}{2}\omega _X$ is nef, so is the line bundle $2\Ll+\omega _X$. In this case, there exists a curve $C\subset X$ such that $C\cdot (2\Ll+\omega _X) =0$. By \cite[Theorem 1-4-8]{matsuki}, $X$ is a $\PP^1$-bundle over a smooth curve $D$, $\pi: X\to D$, and $\Ll\cong\Oo _X(1)$ has degree $1$ on each fiber $C$ of $\pi$. In such a case, $\omega _X(2)$ is nef and hence $\omega_X\cdot \Oo_X(1) \geq -2$. This proves statement (ii). \\
\vspace{2mm}
\item $u\ge 2$ or $v=1$. Then $\omega _X^{\otimes 2}(3)$ is nef and therefore $\omega_X\cdot \Oo_X(1)\geq -3/2$. In this case,
we obtain (iv). 
\end{itemize}
In conclusion, the discussion above yields $\omega_X\cdot \Oo_X(1)\geq -3$ for all $X$, with equality if and only if $X\cong \PP^2$. This proves statement (iii).\end{proof}

\begin{remark}\label{ma2}
Let $n, s$ be integers as in the assumptions of Theorem \ref{ma1}, and let $k$ be an integer.  Assume $\binom{n+s}{n} \le 1+(s^2+3s)k^2/2$. From the ideal sheaf exact sequence and Theorem \ref{ma1} (iii), one has 
\[
h^1(\Ii _X(s)) \ge 1+(s^2+3s)k^2/2 -\binom{n+s}{n},
\]
for any isomorphic projection $X\subset \PP^n$ of the degree $k$ Veronese embedding of $\PP^2$. Assuming the existence of such an $X$ with the additional property that the latter inequality is an equality, one obtains $h^0(\Ii_X(s))=0$. Thus $d_{n,s+1,2} \le \deg(X) = k^2$. 
\end{remark}

When $s=2$, one has the following result that was shown in \cite{bea}: 

\begin{proposition}[{\bf \cite[Theorem 2]{bea}}]\label{be1}
Let $d \ge 2$ and $n\ge 5$. Let $X\subset \PP^n$ be a general isomorphic linear projection of
a degree $k$ Veronese embedding $Y\subset \PP^N$ of $\PP^2$, where $N= (k^2+3k)/2$. Then 
\[
h^0(\Ii _X(2)) = \max \left\{0,\binom{n+2}{2} -\binom{2k+2}{2}\right\},
\]
\[
\mbox{ and } h^1(\Ii _X(2)) =  \max \left\{0,\binom{2k+2}{2}-\binom{n+2}{2} \right\}.
\]
\end{proposition}
An immediate consequence of Proposition \ref{be1} and Remark \ref{ma2} is

\begin{corollary}\label{be3}
Let $n\ge 5$ be an odd integer. Then $d_{n,3,2} \le (n+1)^2/4$.
\begin{proof}
Let $X\subset \PP^n$ be a general linear projection of the degree $(n+1)/2$ Veronese embedding of $\PP^2$. By Proposition
\ref{be1}, we have
\[
h^0(\Ii _X(2)) =\max \left\{0,\binom{n+2}{2} -\binom{n+3}{2}\right\}=0.
\]
\end{proof}
\end{corollary}

\begin{question}
Fix integers $d \ge 2, s\ge 3$, and $n\ge 5$. Let $X\subset \PP^n$ be a general isomorphic linear projection of
a degree $k$ Veronese embedding $Y\subset \PP^N$ of $\PP^2$, where $N= (k^2+3k)/2$. Is it true that
\[
h^0(\Ii _X(s)) = \max \left\{0,\binom{n+s}{n} -\binom{sk+2}{2}\right\},
\]
\[
\mbox{ and } h^1(\Ii _X(s)) =  \max \left\{0,\binom{sk+2}{2}-\binom{n+s}{n} \right\}?
\]

\end{question}

\begin{theorem}\label{be2}
Let $n\ge 6$ be even. Then $h^0(\Ii _X(2)) \ge \binom{n+2}{2} - 1 - 5d$ for all smooth connected and non-degenerate surfaces of degree $d$. 
Moreover, there exists a smooth and connected surface $X\subset \PP^n$ such that $\deg (X) =n^2/4$, $h^0(\Ii _X(2)) =0$, and $X\cong \PP^2$. 
\end{theorem}
\begin{proof}
The first part is statement (iii) of Theorem \ref{ma1}. By Proposition \ref{be1}, the general isomorphic linear projection of the degree $k$ Veronese surface has 
\[
h^0(\Ii _X(2)) =\max \left\{0,\binom{n+2}{2} -\binom{2k+2}{2}\right\}.
\] 
Thus the choice $k = n/2$ establishes the second statement. 
\end{proof}

\begin{remark}\label{w2}
Let $X\subset \PP^n$ be an integral and non-degenerate surface of sectional genus $g$ and degree $d$. Then $h^0(\Ii _{X}(2)) \le \binom{n+1}{2} - 2d-1+g$.
\begin{proof}
Let $H\subset \PP^n$ be a general hyperplane. Consider the exact sequence
\begin{equation}\label{eqw1}
0 \longrightarrow \Ii _X(s-1)\longrightarrow \Ii _X(s)\longrightarrow \Ii _{X\cap H,H}(s)\longrightarrow 0.
\end{equation}
Since $H$ is general, $C = X\cap H$ is a smooth curve of degree $d$ and genus $g$. Hence, by Riemann-Roch, $h^0(\Oo_C(2)) \geq 2d+1-g$. 
Thus 
\[
h^0(H,\Ii _{X\cap H,H}(2)) \le \binom{n+1}{2} - 2d-1+g.
\] 
By (\ref{eqw1}) and $h^0(\Ii _X(1)) =0$, the conclusion follows.
\end{proof}
\end{remark}

\begin{remark}\label{w3}
A complete classification of surfaces  of almost-minimal degree, i.e., all integral and non-degenerate surfaces $X\subset \PP^n$ such that $\deg (X) =n$ is known; see, e.g., \cite{hsv, park1}. Let $H\subset \PP^n$ be a general hyperplane. Since $C=X\cap H$ is a degree $n$ integral curve spanning $H$, there are two possibilities: either $C$ is arithmetically Cohen-Macaulay with $p_a(C) =1$ or $C$ is a smooth rational curve. In the latter case, $X$ may have only finitely many singular points. 

\begin{enumerate}
\item[(i)] $C$ is arithmetically Cohen-Macaulay with $p_a({C}) =1$. The surface $X$ is linearly normal of degree $\deg(X)\geq 2p_a(C)+1$. Note that this case may occur for all $n$: take a cone over a linearly normal elliptic curve $C\subset \PP^{n-1}$. 
\vspace{3mm}
\item[(ii)] $C$ is a smooth rational curve. Thus $h^1(H,\Ii _{X\cap H,H}(1)) =1$, $h^1(H,\Ii _{X\cap H,H}(2)) =0$, 
and so $h^0(H,\Ii _{X\cap H,H}(2)) = \binom{n+1}{2} -2n-1$. It is clear that
\[
h^1(\Ii _X(1)) \le h^1(H,\Ii _{X\cap H,H}(1)) =1.
\]
From (\ref{eqw1}), we derive $\binom{n+1}{2} -2n-2\le h^0(\Ii _X(2)) \le \binom{n+1}{2} -2n-1$. Now, if $X$ is smooth and $X\cap H$ is rational, then the classification of surfaces gives that $X$ is rational. Thus $h^1(\Oo _X)=0$. A standard exact sequence gives $h^0(\Oo _X(1)) =h^0(\Oo _{X\cap H}(1)) +1=n+2$ and $h^0(\Oo _X(2))= h^0(\Oo _{X\cap H}(2)) +h^0(\Oo _X(1)) = 3n+3$. Thus $h^1(\Ii _X(1)) =1$ and so $X$ is not linearly normal. Let $X'\subset \PP^{n+1}$ be a smooth surface such that $X$ is an isomorphic linear projection of $X'$ from some $p\in \PP^{n+1}$ with $p$ not contained in the secant variety of $X'$. Thus $\deg(X') = n$ and since $X'\subset \PP^{n+1}$, it is either a minimal degree surface scroll or (when $n=4$) the Veronese surface. In any case, one knows the minimal free resolution of $X'$, and it follows that $X'$ has property $K_2$, following Alzati and Russo \cite[Definition 3.1]{ar}.  By \cite[Theorem 3.2 or Corollary 3.3]{ar}, we have $h^1(\Ii _X(2))=0$. Hence $h^0(\Ii _X(2)) =\binom{n+2}{2} -3n-3$.
\end{enumerate}

\end{remark}

\begin{small}

\end{small}


\begin{thebibliography}{99}

\bibitem{ar} A. Alzati and F. Russo, {\it On the $k$-normality of projected algebraic varieties}, Bull. Braz. Math Soc, New Series 33(1) (2007), no. 3,  27--48. 

\bibitem{acgh} E. Arbarello, M. Cornalba, P. Griffiths, and J.~Harris, {\it Geometry of Algebraic Curves}, Volume I, Grundlehren der mathematischen Wissenschaften {\bf 267}, Springer-Verlag, New York, 1985.

\bibitem{bbem} E. Ballico, G. Bolondi, P. Ellia and R. M. Mir\`{o}-Roig, {\it Curves of maximum genus in the range A and stick-figures},
Trans. Amer. Math. Soc. 349 (1997), no. 11, 4589--4608.

\bibitem{bea} E. Ballico and P. Ellia, {\it On projections of ruled and Veronese varieties}, J. Algebra 121 (1989), 477--487. 
\bibitem{be2} E. Ballico and P. Ellia, {\it On postulation of curves in $\mathbb {P}^4$}, Math. Z. 188 (1985), 215--223.
\bibitem{be4} E. Ballico and P. Ellia, {\it The maximal rank conjecture for non-special curves in ${\bf {P}}^3$}, Invent. Math.
79 (1985), 541--555.
\bibitem{be3} E. Ballico and P. Ellia, {\it Beyond the maximal rank conjecture for curves in $\mathbb {P}^3$}, in: Space Curves,
Proceedings Rocca di Papa, pp. 1--23, Lecture Notes in Math. 1266, Springer, Berlin, 1985.
\bibitem{be5} E. Ballico and P. Ellia, {\it The maximal rank conjecture for non-special curves in $\PP^n$}, Math. Z. 196 (1987),
355--367.
\bibitem{be6} E. Ballico and P. Ellia, {\it The maximal genus of space curves in the Range A}, preprint at \texttt{arXiv:1811.08807}, 2018. 
\bibitem{bef} E. Ballico, P. Ellia and C. Fontanari, {\it Maximal rank of space curves in the Range A}, Eur. J. Math. 4 (2018),
778--801.
\bibitem{bs} M. Brodmann and P. Schenzel, {\it Arithmetic properties of projective varieties of almost minimal degree}, J. Algebraic Geom. 16 (2007), 347--400.
\bibitem{cc} L. Chiantini and C. Ciliberto, {\it Towards a Halphen theory of linear series on curves}, Trans. Amer. Math. Soc. 351 , no. 6, 2197--2212, 1999. 
\bibitem{df} V. Di Gennaro and D. Franco, {\it Refining Castelnuovo-Halphen bounds}, Rend. Circ. Mat. Palermo (2) 61, no. 1, 91--106, 2012. 
\bibitem{fl1} G. Fl{\o}ystad, {\it Construction of space curves with good properties}, Math. Ann. 289 (1991), no. 1, 33--54.
\bibitem{fl2} G. Fl{\o}ystad, {\it On space curves with good cohomological properties}, Math. Ann. 291 (1991), no. 3, 505--549.
\bibitem{fu} T. Fujisawa, {\it On non-rational numerical Del Pezzo surfaces}, Osaka J. Math. 32 (1995), 613--636.
\bibitem{gp1} L. Gruson and C. Peskine, {\it Genre des courbes de l'espace projectif}, Proc. Troms\o \  1977, 39--59, Lect. Notes in
Math. 687, 1978, Springer, Berlin.
\bibitem{glp} L. Gruson, R. Lazarsfeld, and C. Peskine, {\it On a theorem of Castelnuovo and the equations defining space curves}, Inv. Math. 72, 491--506, 1983.
\bibitem{harris81} J. Harris, {\it A bound on the geometric genus of projective varieties}, Ann. Sc. Norm. Super. Pisa Cl. Sci. (4) {\bf 8}, no. 1, 35--68, 1981. 
\bibitem{h} R. Hartshorne, {\it Algebraic Geometry}, Springer-Verlag, Berlin--Heidelberg--New York, 1977.
\bibitem{h0} R. Hartshorne, {\it On the classification of algebraic space curves}, in: Vector bundles and differential equations (Nice 1979), p. 82--112, Progress in Math. 7 Birkh\"{a}user, Boston 1980.
\bibitem{hh} R. Hartshorne and A. Hirschowitz, {\it Nouvelles courbes de bon genre dans l'espace projectif},  Math. Ann. 280 (1988), no. 3, 353--367. 
\bibitem{hi} A. Hirschowitz, {\it Sur la postulation g\'{e}n\'{e}rique des courbes rationnelles}, Acta Math. 146 (1981),
209--230.
\bibitem{hsv} L-T. Hoa, J. St\"{u}ckrad, W. Vogel, {\it Towards a structure theory for projective varieties of degree $=$ codimension $+ 2$}, J. Pure Appl. Algebra 71 (1991) 203--231.
\bibitem{laz87} R. Lazarsfeld, {\it A sharp Castelnuovo bound for smooth surfaces}, Duke Math. J. 55 (1987), 423--429. 
\bibitem{Lvov96} S. L'vovsky, {\it On the inflection points, monomial curves, and hypersurfaces containing projective curves}, Math. Ann. {\bf 306} (1996), 719--735. 
\bibitem{ms} E. Macr\`{\i} and B. Schmidt, {\it Derived categories and the genus of space curves}, preprint at \texttt{arXiv:1801.02709}, 2018. 
\bibitem{matsuki} K. Matsuki, {\it Introduction to the Mori program}, Springer, Berlin, 2002.
\bibitem{park1} E. Park, {\it Smooth varieties of almost minimal degree}, J. Algebra 314 (2007) 185--208.
\bibitem{park} E. Park, {\it On hypersurfaces containing projective varieties}, Forum Math. 27 (2015), no. 2, 843--875.
\bibitem{szp} L. Szpiro, {\it Lectures on Equations defining Space Curves}, Tata Institute of Fundamental Research, Bombay, 1979.
\end{thebibliography}
\end{document}